\documentclass[10pt]{amsart}
\usepackage{amsmath}
\usepackage[usenames,dvipsnames]{color}
\usepackage{parskip}
\usepackage{amsfonts}
\usepackage{amscd}
\usepackage[centertags]{amsmath}
\usepackage{amssymb}
\usepackage{stmaryrd}
\usepackage[all,cmtip]{xy}
\usepackage[english]{babel}
\usepackage{tabularx}
\usepackage{amsxtra}
\usepackage{euscript}
\usepackage[T1]{fontenc}
\usepackage{doc, exscale, fontenc, latexsym, syntonly}
\usepackage{amsfonts}
\usepackage{amsthm}
\usepackage{graphicx}
\usepackage{mciteplus}
\usepackage{cite}

\newtheorem{theorem}{Theorem}[section]
\newtheorem{corollary}[theorem]{Corollary}
\newtheorem{lemma}[theorem]{Lemma}
\newtheorem{proposition}[theorem]{Proposition}
\newtheorem{remark}[theorem]{Remark}

\theoremstyle{definition}
\newtheorem{definition}[theorem]{Definition}
\newtheorem{example}[theorem]{Example}
\theoremstyle{remark}

\numberwithin{figure}{section}
\numberwithin{table}{section}
\newcommand*\acknowledgment[1]{%
	\begingroup\noindent
	\rightskip\leftskip
	\begin{flushleft}\textbf{\large Acknowledgment.}\, #1%
		\par\vspace*{1mm}\end{flushleft}\endgroup}
\begin{document}

\title[Higher Topological Complexity For Fibrations]{Higher Topological Complexity For Fibrations}

\author{MEL\.{I}H \.{I}S and \.{I}SMET KARACA}
\date{\today}

\address{\textsc{Melih Is}
Ege University\\
Faculty of Sciences\\
Department of Mathematics\\
Izmir, Turkey}
\email{melih.is@ege.edu.tr}
\address{\textsc{Ismet Karaca}
Ege University\\
Faculty of Science\\
Department of Mathematics\\
Izmir, Turkey}
\email{ismet.karaca@ege.edu.tr}



\begin{abstract}
We introduce the higher topological complexity (TC$_{n}$) of a fibration in two ways: the higher homotopic distance and the Schwarz genus. Then we have some results on this notion related to TC, TC$_{n}$ or cat of a topological space or a fibration. We also show that TC$_{n}$ of a fibration is a fiber homotopy equivalence.
\end{abstract}

\maketitle

%%%%%%%%%%%%%%%%%%%%%%%%%%%%%%%%%%%%%%%%%%%%%%%%%%%%%
\section{Introduction}
\label{intro}
\quad The works on determining the topological complexity have been developing since Farber first computed the number of the topological complexity (TC) for a topological space \cite{Farber:2003} (see \cite{Farber:2008} for a comprehensive analysis of the investigation of this concept). Note that such a space must be path-connected. Rudyak \cite{Rudyak:2010} has the number of higher topological complexity TC$_{n}$ for a topological space for the first time (similarly the space is path-connected). It is a reasoned generalization on existing works of TC because of the fact that TC$_{2} = $ TC. After that Pavesic \cite{Pavesic:2019} gives a new point of view to the developing explorations by determining the topological complexity for a map. Such a map must be surjective and continuous or one can choose a surjective fibration. These studies have a common point: All of these investigations can be considered by using the concept of the Schwarz genus (secat) for a fibration, introduced by Schwarz \cite{Schwarz:1966}. One can use a fibrational substitute, as well. On the other hand, there is one more option, which is the homotopic distance \cite{VirgosLois:2020}. It can be used for the computations of TC or TC$_{n}$ of a topological space or a map. The higher homotopic distance \cite{BorVer:2021} is a general case of the homotopic distance and especially useful for the studies of TC$_{n}$ of a topological space. Since LS-category (cat) for a topological space or a map is the base point for the research of TC, we obtain cat by using both the Schwarz genus and the homotopic distance \cite{CorneaLupOpTan:2003,VirgosLois:2020}. 

\quad These studies lead us to solve an open problem and introduce a new concept the higher topological complexity of a surjective fibration $f$ between any path-connected topological spaces, denoted by TC$_{n}(f)$. For $n=2$, we show that TC$_{2}(f) = $ TC$(f)$. In addition, we have that TC$_{n}(f) =$ TC$_{n}(Y)$ if $f$ is $1_{Y} : Y \rightarrow Y$. Therefore, TC$_{n}(f)$ is a general version of respective notions TC$(Y)$, TC$_{n}$ and TC$(f)$, where $f : X \rightarrow X^{'}$ is a surjective fibration and $Y$ is path-connected. We define TC$_{n}(f)$ by using the two options the Schwarz genus and the higher homotopic distance. 

\quad In this study, we first give explicit definitions and significant results of TC, TC$_{n}$ and their related invariants such as D$(f,g)$ and cat. Then we express TC$(f)$ as the homotopic distance. With this expression, we state some results on TC$(f)$. So, we introduce the definition of TC$_{n}(f)$ for a surjective fibration $f$. Later, we state some properties on TC$_{n}(f)$. In addition, we prove that TC$_{n}(f)$ is an invariant on fiber homotopy equivalence fibrations. We also research the relationship between TC$_{n}(f)$ and the other related notions such as TC$_{n}(X)$ and cat$(X)$. We choose some special fibrations as examples for computing their higher topological complexities. In Section \ref{sec:3}, we define the higher topological complexity of a surjective fibration by using secat. 

\section{Preliminaries}
\label{sec:1}
 \quad We begin with recalling some basic notions and results on the homotopic distance, the topological complexities, and the Lusternik-Schnirelmann category.
 
 \begin{definition}\cite{VirgosLois:2020}
 	Given two continuous functions $f$, $f^{'} : X \rightarrow X^{'}$. Then the \textit{homotopic distance} for $f$ and $f^{'}$, denoted by $D(f,f^{'})$, is the smallest integer $l \geq 0$ satisfying that there is a covering
 	\begin{eqnarray*}
 		X = V_{1} \cup \cdots \cup V_{l}
 	\end{eqnarray*}
 	and the restrictions of $f$ and $f^{'}$ on $V_{j}$, $j = 1, \cdots, l$, are homotopic, i.e., $f_{|_{V_{j}}} \simeq f^{'}_{|_{V_{j}}}$. 
 \end{definition}
 
 \begin{definition}\cite{VirgosLois:2020,BorVer:2021}\label{d3}
 	Let $f_{i} : X \rightarrow X^{'}$ be any continuous map for $i = 1, \cdots, k$. Then the \textit{higher homotopic distance}, denoted by $D(f_{1}, \cdots, f_{k})$, is the smallest positive integer $l$ for which
 	\begin{eqnarray*}
 		X = U_{1} \cup \cdots \cup U_{l}
 	\end{eqnarray*}
    is a covering and the condition
    \begin{eqnarray*}
    	f_{1}^{j}|_{U_{j}} \simeq  \cdots \simeq  f_{k}^{j}|_{U_{j}}
    \end{eqnarray*}
    satisfies for each $j = 1, \cdots, l$. If such a covering does not exist, then $D(f_{1}, \cdots, f_{k})$ equals $\infty$.   
 \end{definition}

\quad Note that, in the definition of (higher) homotopic distance, the covering of $X$ has $k+1$ open sets $U_{0}, U_{1}, \cdots U_{k}$, but we consider that $U_{1}$ is the first open set in this covering for consistency of Farber's topological complexity definition (see \cite{Farber:2003}). 

\begin{proposition}\cite{VirgosLois:2020}\label{p8}
	    Each of the following conditions on the notion homotopic distance is satisfied:
	    
		\textbf{i)} If $f$ and $g$ are homotopic ($f \simeq g$), and $f^{'}$ and $g^{'}$ are homotopic ($f^{'} \simeq g^{'}$), then \[\text{D}(f,f^{'}) = \text{D}(g,g^{'}).\]
		
		\textbf{ii)} For any functions $f$, $f^{'} : X \rightarrow X^{'}$, and $k : Z \rightarrow X$, we have \[\text{D}(f \circ k,f^{'} \circ k) \leq \text{D}(f,f^{'}).\]
		
		\textbf{iii)} Let $f$, $f^{'} : X \rightarrow X^{'}$ be any functions. If $k : X^{'} \rightarrow Z$ is a function such that there is a function $k^{'} : Z \rightarrow X^{'}$ satisfying $k^{'} \circ k \simeq 1_{X^{'}}$, then \[\text{D}(k \circ f, k \circ f^{'}) = \text{D}(f,f^{'}).\]
		
		\textbf{iv)} Let $f$, $f^{'} : X \rightarrow X^{'}$ be any functions. If $k : Z \rightarrow X$ is a function such that there is a function $k^{'} : X \rightarrow Z$ satisfying $k \circ k^{'} \simeq 1_{X}$, then \[\text{D}(f \circ k, f^{'} \circ k) = \text{D}(f,f^{'}).\]
		
		\textbf{v)} Let $X$ be normal. If $f$, $f^{'} : X \rightarrow X^{'}$, and $g$, $g^{'} : X^{'} \rightarrow Y$ are any functions, then \[\text{D}(f \circ g,f^{'} \circ g^{'}) \leq \text{D}(f,f^{'}) + \text{D}(g,g^{'}).\]
\end{proposition} 

\begin{proposition}\cite{BorVer:2021}\label{p9}
	Each of the following conditions on the notion higher homotopic distance is satisfied:
	
	\textbf{i)} If $f_{1}, f_{2}, \cdots, f_{l}, \cdots, f_{k} : X \rightarrow X^{'}$ are any functions for $1 < l < k$, then \[\text{D}(f_{1}, \cdots, f_{l}) \leq \text{D}(f_{1}, \cdots, f_{k}).\]
	
	\textbf{ii)} If $f_{1}, \cdots, f_{k} : X \rightarrow X^{'}$ and $g_{1}, \cdots, g_{k} : X^{'} \rightarrow Y$ are functions such that $g_{i} \simeq g_{i+1}$ for every $i = 1, \cdots, k-1$, then \[\text{D}(g_{1} \circ f_{1}, \cdots, g_{k} \circ f_{k}) \leq \text{D}(f_{1}, \cdots, f_{k}).\]
	
	\textbf{iii)} Given two path-connected spaces $X$ and $X^{'}$. If $f_{1}, \cdots, f_{k} : X \rightarrow X^{'}$ are any functions, then \[\text{D}(f_{1}, \cdots, f_{k}) \leq \text{cat}(X).\]
	
	\textbf{iv)} If $f_{1}, \cdots, f_{k} : X \rightarrow X^{'}$ are any functions, then \[\text{D}(f_{1}, \cdots, f_{k}) \leq \text{TC}_{k}(X^{'}).\]
	
	\textbf{v)} Let $X \times \overline{X^{'}}$ be normal. If $f_{1}, \cdots, f_{k} : X \rightarrow X^{'}$ and $g_{1}, \cdots, g_{k} : \overline{X} \rightarrow \overline{X^{'}}$ are any functions, then
	\begin{eqnarray*}
		\text{D}(f_{1} \times g_{1}, \cdots, f_{k} \times g_{k}) \leq \text{D}(f_{1}, \cdots, f_{k}) + \text{D}(g_{1}, \cdots, g_{k}).
	\end{eqnarray*}
\end{proposition} 

\quad If $f : X \rightarrow Z$ is a fibration, then the \textit{Schwarz genus} \cite{Schwarz:1966} of it, denoted by $secat(f)$, is the smallest integer $l \geq 0$ for which
\begin{eqnarray*}
	Z = V_{1} \cup \cdots \cup V_{l}
\end{eqnarray*}
is a covering and the condition
\begin{eqnarray*}
	f \circ s_{i} = 1_{V_{i}}
\end{eqnarray*}
satisfies for each $V_{i}$, $i = 1, \cdots, l$. If there is no such a covering of $Z$, then we point out that $secat(f) = \infty$. We again note that we use $l$ open sets in the covering of $Y$ instead of $l+1$ for consistency of our results with Definition \ref{d3}. Just as homotopic distance, the notion Schwarz genus is an important common point in the exploration of motion planning problems, because each of TC, TC$_{n}$, and cat is motivated by the Schwarz genus of some special fibrations.
    
\begin{definition}\cite{Farber:2003}\label{d4}
	For any path-connected $Y$, the \textit{topological complexity} of $Y$, denoted by $TC(Y)$, is $secat(\pi)$, where $\pi : Y^{I} \rightarrow Y \times Y$, $\pi(\beta) = (\beta(0),\beta(1))$, is a path fibration.
\end{definition}

\begin{definition}\cite{Rudyak:2010}\label{d5}
	For any path-connected $Y$, if $\beta = (\beta_{1}, \cdots, \beta_{n}) : I \rightarrow Y^{n}$ is a multipath in $Y$ such that the starting point of each $\beta_{i}$, $i = 1, \cdots, n$, is the same. Then the \textit{higher topological complexity} of $Y$, denoted by $TC_{n}(Y)$, is $secat(e_{n})$, where $e_{n} : (Y^{n})^{I} \rightarrow Y^{n}$, $\beta \mapsto e_{n}(\beta) = (\beta_{1}(1), \cdots, \beta_{n}(1))$, is a fibration.
\end{definition}

\quad Denote $J_{n}$ by the wedge of $n$ copies of $[0,1]$ such that the initial point $0_{i}$ is the same point for each $i = 1, \cdots, n$. Then the function space $Y^{J_{n}}$ includes all multipaths in a path-connected space $Y$. Therefore, $TC_{n}(Y)$ is also defined as $secat(e_{n} : Y^{J_{n}} \rightarrow Y_{n})$, $e_{n}(\beta) = (\beta_{1}(1), \cdots, \beta_{n}(1))$ for a multipath $\beta = (\beta_{1}, \cdots, \beta_{n}) \in Y^{J_{n}}$. $TC_{n}(Y) = 1$ when $n = 1$. For $n=2$, 
Definition \ref{d4} coincides with Definition $\ref{d5}$. Moreover, $TC_{n}(Y) \leq TC_{r}(Y)$ if $n+1 \leq r$.
\begin{definition}\cite{Pavesic:2019}\label{d6}
	Let $g : Y \rightarrow Y^{'}$ be a surjective fibration. Then the \textit{topological complexity of} $g$, denoted by $TC(g)$, is $secat(\pi_{g})$, where $\pi_{g} : Y^{I} \rightarrow Y \times Y^{'}$ with $\pi_{g}(\beta) = (\beta(0),g \circ \beta(1))$ is a fibration.
\end{definition}
\quad For a fibration $g$, $TC(g)$ is also motivated by the homotopic distance \cite{VirgosLoisSaez:2021}. We observe the same thing with an alternative method in Theorem \ref{t2}, as well.

\begin{proposition} \cite{Pavesic:2019}\label{p7}
	The following hold:
	
	\textbf{i)} Let $g : Y \rightarrow Y^{'}$ be a surjective and continuous function. If $k : Z \rightarrow Y$ is a function such that there is a function $k^{'} : Y \rightarrow Z$ satisfying $k \circ k^{'} \simeq 1_{Y}$, then \[\text{TC}(g \circ k) \geq \text{TC}(g).\]
	
	\textbf{ii)} Let $g : Y \rightarrow Y^{'}$ be a surjective and continuous function. If $k : Z \rightarrow Y$ is a function such that there is a function $k^{'} : Y \rightarrow Z$ satisfying $k^{'} \circ k \simeq 1_{Z}$, then \[\text{TC}(g \circ k) \leq \text{TC}(g).\]
	
	\textbf{iii)} If $g : Y \rightarrow Y^{'}$ is a surjective and continuous map, and $k : Z \rightarrow Y$ is a fibration, then \[\text{TC}(g \circ k) \leq \text{TC}(g).\] 
\end{proposition}

\quad Recall that two functions $g : Y \rightarrow Y^{'}$ and $g^{'} : Z \rightarrow Y^{'}$ are \textit{fiber homotopy equivalent} (FHE-equivalent) if the diagram 
\begin{displaymath}
\xymatrix{
	Y \ar[dr]_{g} \ar@<1ex>[rr]^u
	& & Z \ar@<1ex>[ll]^v \ar[dl]^{g^{'}} \\
	& Y^{'} & }
\end{displaymath}
is commutative with $u \circ v \simeq 1_{Z}$ and $v \circ u \simeq 1_{Y}$.

\begin{theorem} \cite{Pavesic:2019}\label{t4}
	Let $g : Y \rightarrow Y^{'}$ and $g^{'} : Z \rightarrow Y^{'}$ be maps. Given any fibrewise maps $u : Y \rightarrow Z$ and $v : Z \rightarrow Y$ with $u \circ v \simeq 1_{Z}$ and $v \circ u \simeq 1_{Y}$, we have \[\text{TC}(g) = \text{TC}(g^{'}).\]
\end{theorem}

\quad Denote $P_{0}Y$ as the space of entire paths in a topological space $Y$ such that the starting points of the paths are the same point. Then the \textit{LS-category} \cite{CorneaLupOpTan:2003} of $Y$, denoted by $cat(Y)$, is $secat(\pi_{Y})$, where $\pi_{Y} : P_{0}Y \rightarrow Y$, $\pi_{Y}(\beta) = \beta(1)$, is a fibration. If $g : Y \rightarrow Y^{'}$ is a function, then the \textit{LS-category} \cite{CorneaLupOpTan:2003} of $g$, denoted by $cat(g)$, is the smallest integer $m \geq 0$ for which
\begin{eqnarray*}
	Y = W_{1} \cup \cdots \cup W_{m}
\end{eqnarray*}
is a covering such that $g|_{W_{i}}$ is nullhomotopic for each $i = 1, \cdots, m$. In a similar manner, for consistency, we use $m$ open sets for the cover of $X$ instead of $m+1$ open sets in the definition of $cat(g)$.

\quad Using the higher homotopic distance; alternative expressions of $TC(Y)$, $TC_{n}(Y)$, $TC(g)$, $cat(Y)$ and $cat(g)$ are given by the next Theorem:

\begin{theorem} \cite{VirgosLois:2020,BorVer:2021,VirgosLoisSaez:2021}
Each of the following holds:

\textbf{i)} Let $Y$ be path-connected. Then \[\text{TC}(Y) = \text{D}(p_{1},p_{2}),\] where $p_{1}, p_{2} : Y \times Y \rightarrow Y$ are functions with $p_{1}(y,y^{'}) = y$ and $p_{2}(y,y^{'}) = y^{'}$, respectively.

\textbf{ii)} Let $Y$ be path-connected. Then \[\text{TC}_{n}(Y) = \text{D}(p_{1}, \cdots, p_{n}),\] where $p_{i} : Y^{n} \rightarrow Y$ is the projection for each $j \in \{1, \cdots, n\}$.

\textbf{iii)} If $g : Y \rightarrow Y^{'}$ is a surjective fibration, then \[\text{TC}(g) = \text{D}(g \circ \pi_{1},\pi_{2}),\] where $\pi_{1} : Y \times Y^{'} \rightarrow Y$ and $\pi_{2} : Y \times Y^{'} \rightarrow Y$ are projection maps.

\textbf{iv)} For any space $Y$,
\[\text{cat}(Y) = \text{D}(j_{1},j_{2}),\] 
where $j_{1} : Y \rightarrow Y \times Y$ is defined with $j_{1}(y) = (y,y_{0})$ and $j_{2} : Y \rightarrow Y \times Y$ is defined with $j_{2}(y) = (y_{0},y)$ for the point $y_{0} \in Y$.

\textbf{v)} If $g : Y \rightarrow Y^{'}$ is a function such that $Y^{'}$ is path-connected, then \[\text{cat}(g) = \text{D}(g,\ast),\] where $\ast$ is a constant map.
\end{theorem}

\section{The Higher Topological Complexity Using The Homotopic Distance}
\label{sec:2}

\quad Let $f : X \rightarrow Y$ be a fibration, $p_{1} : Y \times Y \rightarrow Y$ be a projection defined by $p_{1}(y,y^{'}) = y$, and $\pi_{2} : X \times Y \rightarrow Y$ be another projection defined by $\pi_{2}(x,y) = y$. In Theorem 2.8 of \cite{VirgosLois:2020}, if we take $X \times Y$, $p_{1} \circ (f \times 1_{Y})$, and $\pi_{2}$ instead of $X$, $f$ and $g$, respectively, then we find that 
\begin{displaymath}
\xymatrix{
	P \ar[r] \ar[d]_{q_{1}} &
	PY \ar[d]^{\pi} \\
	X \times Y \ar[r]_{h} & Y \times Y}
\end{displaymath}
is commutative, where $h$ is the map $(p_{1} \circ (f \times 1_{Y}),\pi_{2}) : X \times Y \rightarrow Y \times Y$. This proves that
\begin{equation}
	D(p_{1} \circ (f \times 1_{Y}),\pi_{2}) = secat(q_{1}).
\end{equation}

\quad On the other hand, in Theorem 4.9 of \cite{Pavesic:2019}, if we take $p_{1} \circ (f \times 1_{Y})$ instead of the fibration $f$, then we observe that
\begin{displaymath}
\xymatrix{
	X \sqcap PY \ar[r] \ar[d]_{q_{2}} &
	PY \ar[d]^{\pi} \\
	X \times Y \ar[r]_{k} & Y \times Y}
\end{displaymath}
is also commutative, where $k$ is $(p_{1} \circ (f \times 1_{Y})) \times 1_{Y} : X \times Y \rightarrow Y \times Y$. This gives 
\begin{equation}
	TC(p_{1} \circ (f \times 1_{Y})) = secat(q_{2}).
\end{equation}

\begin{theorem} \label{t1}
	For a fibration $f : X \rightarrow Y$, if $p_{1} : Y \times Y \rightarrow Y$ and $\pi_{2} : X \times Y \rightarrow Y$ are projections with $p_{1}(y,y^{'}) = y$ and $\pi_{2}(x,y) = y$, respectively, then
	\begin{eqnarray*}
		\text{D}(p_{1} \circ (f \times 1_{Y}),\pi_{2}) = \text{TC}(p_{1} \circ (f \times 1_{Y})).
	\end{eqnarray*}
\end{theorem}
\vspace*{-0.6cm}
\begin{proof}
	By (1) and (2), it is enough to show that $q_{1}$ and $q_{2}$ are fibre homotopy equivalent fibrations. Define the maps $u : P \rightarrow X \sqcap PY$ and $v : X \sqcap PY \rightarrow P$ with $u(x,y,\alpha) = (x,\alpha)$ and $v(x,\alpha) = (x,\alpha(1),\alpha)$, respectively. Since $P \subset (X \times Y) \times PY$ is given by $\{((x,y),\alpha) \ \ : \ \ \alpha(0) = p_{1} \circ (f \times 1_{Y})(x,y), \alpha(1) = \pi_{2}(x,y)\}$, we have that $\alpha(0) = f(x)$ and $\alpha(1) = y$. It follows that
	\begin{eqnarray*}
		u \circ v = 1_{X \sqcap PY} \ \ \ \text{and} \ \ \ v \circ u = 1_{P}.
	\end{eqnarray*}
    Thus, $q_{1}$ and $q_{2}$ are fibre homotopy equivalent to each other.
\end{proof}

\begin{corollary}
		For a fibration $f : X \rightarrow Y$, if $p_{1} : Y \times Y \rightarrow Y$ and $\pi_{2} : X \times Y \rightarrow Y$ are projections with $p_{1}(y,y^{'}) = y$ and $\pi_{2}(x,y) = y$, respectively, then
	\begin{eqnarray*}
		\text{D}(p_{1} \circ (f \times 1_{Y}),\pi_{2}) \leq \text{TC}(Y).
	\end{eqnarray*}
    Moreover, if $Y$ is contractible, then we conclude that $D(p_{1} \circ (f \times 1_{Y}),\pi_{2}) = 1$.
\end{corollary}
    
\begin{proof}
	By Theorem \ref{t1}, we obtain
	\begin{eqnarray*}
		D(p_{1} \circ (f \times 1_{Y}),\pi_{2}) = TC(p_{1} \circ (f \times 1_{Y})).
	\end{eqnarray*}
	Since $f$ is a fibration, so is $f \times 1_{Y}$. Using Proposition \ref{p7} (iii), and Example 4.10 in \cite{Pavesic:2019}, we get \begin{eqnarray*}
		TC(p_{1} \circ (f \times 1_{Y})) \leq TC(p_{1}) = TC(Y).
	\end{eqnarray*}
	Finally, the contractibility of $Y$ implies that $D(p_{1} \circ (f \times 1_{Y}),\pi_{2}) = 1$. 
\end{proof}

\begin{theorem}\label{t2}
	For a fibration $f : X \rightarrow Y$, if $\pi_{1} : X \times Y \rightarrow X$ and $\pi_{2} : X \times Y \rightarrow Y$ are projections with $\pi_{1}(x,y) = x$ and $\pi_{2}(x,y) = y$, respectively, we have
	\begin{eqnarray*}
		\text{D}(f \circ \pi_{1},\pi_{2}) = \text{TC}(f).
	\end{eqnarray*}
\end{theorem}

\begin{proof}
	Given a fibration $f : X \rightarrow Y$, we rewrite the continuous function $p_{1} \circ (f \times 1_{Y})$ as $f \circ \pi_{1}$. By Proposition \ref{p8} (i), we have
	\begin{eqnarray*}
		D(p_{1} \circ (f \times 1_{Y}),\pi_{2}) = D(f \circ \pi_{1},\pi_{2}).
	\end{eqnarray*}
    From Theorem 2.8, we find $D(f \circ \pi_{1},\pi_{2}) = TC(f \circ \pi_{1})$. Define a map $u : X \rightarrow X \times Y$ with $u(x) = (x,f(x))$. Therefore, we get $\pi_{1} \circ u \simeq 1_{X}$. By Proposition \ref{p7} (i), we obtain
    \begin{eqnarray*}
    	D(f \circ \pi_{1},\pi_{2}) = TC(f \circ \pi_{1}) \geq TC(f).
    \end{eqnarray*}
    On the other hand, Proposition \ref{p7} (iii)
    gives us that
    \begin{eqnarray*}
    	D(f \circ \pi_{1},\pi_{2}) = TC(f \circ \pi_{1}) \leq TC(f).
    \end{eqnarray*}
    As a consequence, $D(f \circ \pi_{1},\pi_{2}) = TC(f)$.
\end{proof}

\quad Theorem \ref{t2} confirms the following result \cite{Pavesic:2019}:

\begin{corollary}
	The topological complexity of a space coincides with the topological complexity of the identity function on it.
\end{corollary}

\begin{proposition}\label{p1}
	If $f : X \rightarrow X^{'}$ is a fibration, then we have
	\begin{eqnarray*}
		\text{D}(f,f) \leq \text{TC}(f).
	\end{eqnarray*}
\end{proposition}

\begin{proof}
	Consider a fibration $f : X \rightarrow X^{'}$ and a map $h : X \rightarrow X \times X^{'}$ defined as $h(x) = (x,f(x))$. By Proposition \ref{p8} (ii), we obtain
	\begin{eqnarray*}
		TC(f) = D(f \circ \pi_{1},\pi_{2}) \geq D((f \circ \pi_{1}) \circ h,\pi_{2} \circ h) = D(f,f).
	\end{eqnarray*}
\end{proof}

\quad Proposition \ref{p1} also confirms the well-known fact that $TC$ of a fibration \linebreak $f : X \rightarrow X^{'}$ is a nonnegative integer when considering $D(f,f) = 1$. 

\begin{proposition}
	For any fibrations $f$, $f^{'}: X \rightarrow X^{'}$, we have
	 \begin{eqnarray*}
		\text{D}(f,f^{'}) \leq \text{TC}(f) \cdot \text{TC}(f^{'}).
	\end{eqnarray*}
\end{proposition}

\begin{proof}
	Let $TC(f) = k$ and $TC(f^{'}) = l$. Then $\{U_{i}\}_{i=1}^{k}$ and $\{V_{j}\}_{j=1}^{l}$ are corresponding coverings of $X \times X^{'}$. In addition, we get
	\begin{eqnarray*}
		(f \circ \pi_{1})\big|_{U_{i}} \simeq \pi_{2}\big|_{U_{i}} \ \ \ \ \text{and} \ \ \ \ (f^{'} \circ \pi_{1})\big|_{V_{j}} \simeq \pi_{2}\big|_{V_{j}}.
	\end{eqnarray*}
    Let $W_{i,j} = U_{i} \cap V_{j}$. This covers $X \times X^{'}$ and we derive
    \begin{eqnarray*}
    	f \circ \pi_{1}\big|_{W_{i,j}} = (f \circ \pi_{1})\big|_{W_{i,j}} \simeq \pi_{2}\big|_{W_{i,j}} \simeq (f^{'} \circ \pi_{1})\big|_{W_{i,j}} = f^{'} \circ \pi_{1}\big|_{W_{i,j}}.
    \end{eqnarray*}
    This means that $D(f \circ \pi_{1}\big|_{W_{i,j}}, f^{'} \circ \pi_{1}\big|_{W_{i,j}}) \leq k \cdot l$. Recalling that $h : X \rightarrow X \times X^{'}$, $h(x) = (x,f(x))$, is a right homotopy inverse of $\pi_{1}$, by Proposition \ref{p8} (iii), we conclude that
    \begin{eqnarray*}
    	D(f,f^{'}) = D(f \circ \pi_{1}\big|_{W_{i,j}}, f^{'} \circ \pi_{1}\big|_{W_{i,j}}) \leq k \cdot l.
    \end{eqnarray*}
 \end{proof}

\quad Assume that $f = (f_{1}, \cdots, f_{n}) : X \rightarrow Y^{n}$ is a function. If $f$ is a fibration, then $f_{i} : X \rightarrow Y$ is a fibration for all $i = 1, \cdots, n$. Indeed, for the projection map $p_{i} : Y^{n} \rightarrow Y$ onto the $ith-$factor for each $i$, we have $p_{i} \circ f = f_{i}$. Since projections are fibrations, $f_{i}$ is a fibration for all $i$. 

\begin{definition} \label{d1}
	Let $f = (f_{1},f_{2},\cdots,f_{n}) : X \rightarrow Y^{n}$ be a surjective fibration for $n > 1$. Let $p_{i} : X^{n} \rightarrow X$ be the projection map onto the $ith-$factor for $1 \leq i \leq n$. Then the $n-$dimensional higher topological complexity of $f$ is \[TC_{n}(f) = D(f \circ p_{1},f \circ p_{2}, \cdots, f \circ p_{n}).\]
\end{definition}

\quad In Definition \ref{d1}, we assume that $TC_{1}(f)$ is always equal to $1$. If $f$ is an identity map, then $TC_{n}(f) = TC_{n}(X)$. In addition, the following result proves that $TC_{2}(f)$ coincides with the notion $TC(f)$ in Theorem \ref{t2}.

\begin{proposition}\label{p5}
	TC$_{2}(f)$ in Definition \ref{d1} coincides with TC$(f)$ in Theorem $\ref{t2}$.
\end{proposition}

\begin{proof}
	Let $f = (f_{1},f_{2}) : X \rightarrow Y \times Y$ and $f^{'} : X \rightarrow Y$ be any continuous functions. For each projection $p_{i} : X^{n} \rightarrow X$, $i=1,2$, and the functions $\pi_{1} : X \times Y \rightarrow X$ and $\pi_{2} : X \times Y \rightarrow Y$ with $\pi_{1}(x,y) = x$ and $\pi_{2}(x,y) = y$, respectively, we shall prove that $D(f \circ p_{1},f \circ p_{2}) = D(f^{'} \circ \pi_{1},\pi_{2})$. We take $\alpha = 1_{Y} : Y \rightarrow Y$ and $\beta = 1_{X} \times f^{'} : X \times X \rightarrow X \times Y$ and consider the following two diagrams:
	\begin{displaymath}
	\xymatrix{
		X \times Y \ar[r]^{f^{'} \circ \pi_{1}} &
		Y \ar[d]^{1_{Y}} \\
		X \times X \ar[u]^{1_{X} \times f^{'}} \ar[r]_{f \circ p_{1}} & Y,}
	\end{displaymath}
	\begin{displaymath}
	\xymatrix{
		X \times Y \ar[r]^{\pi_{2}} &
		Y \ar[d]^{1_{y}} \\
		X \times X \ar[u]^{1_{X} \times f^{'}} \ar[r]_{f \circ p_{2}} & Y.}
	\end{displaymath}
	There is an element $x^{'} \in X$ with $f^{'}(x^{'}) = y$ because $f^{'}$ is surjective. Define a continuous function $\gamma : X \times Y \rightarrow X \times X$ with $\gamma(x,y) = (x,x^{'})$. Then, we observe that $\gamma \circ \beta$ and $\beta \circ \gamma$ are identity maps, namely that, $\gamma$ is a homotopy inverse of the function $\beta$. By Proposition 3.14 in \cite{VirgosLois:2020}, we conclude that $D(f \circ p_{1},f \circ p_{2})$ equals $D(f^{'} \circ \pi_{1},\pi_{2})$.
\end{proof}

\begin{proposition}\label{p6}
	For a fibration $f : X \rightarrow Y^{n}$, we have TC$_{n}(f) \leq$ TC$_{n+1}(f)$. 
\end{proposition}

\begin{proof}
	Let $TC_{n+1}(f) = k$. Then $D(f \circ p_{1}, \cdots, f \circ p_{n+1}) = k$ for a fibration \linebreak $f : X \rightarrow Y^{n}$. By Proposition \ref{p9} (i), we obtain $D(f \circ p_{1}, \cdots, f \circ p_{n}) \leq k)$. It follows that $TC_{n}(f) \leq k$.
\end{proof}

\begin{proposition}
	If $f = (f_{1}, \cdots, f_{n}) : X \rightarrow Y^{n}$ and $f^{'} = (f^{'}_{1}, \cdots, f^{'}_{n}) : \overline{X} \rightarrow \overline{Y}^{n}$ are two fibrations with a normal space $X \times \overline{X}$, then 
	\begin{eqnarray*}
		\text{TC}_{n}(f \times f^{'}) \leq \text{TC}_{n}(f) + \text{TC}_{n}(f^{'}).
	\end{eqnarray*}
\end{proposition}

\begin{proof}
	Assume that $TC_{n}(f) = k$ and $TC_{n}(f^{'}) = l$. Then
	$D(f \circ p_{1}, \cdots, f \circ p_{n}) = k$ and $D(f^{'} \circ \overline{p}_{1}, \cdots, f^{'} \circ \overline{p}_{n}) = l$ for projections maps $p_{i} : X^{n} \rightarrow X$ and $\overline{p}_{i} : \overline{X}^{n} \rightarrow \overline{X}$ with each $i \in \{1, \cdots, n\}$. By Proposition \ref{p9} (v), we have
	\begin{eqnarray*}
		&&D(f \circ p_{1}, \cdots, f \circ p_{n}) + D(f^{'} \circ \overline{p}_{1}, \cdots, f^{'} \circ \overline{p}_{n}) \\
		&\geq&D((f \circ p_{1}) \times (f^{'} \circ \overline{p}_{1}), \cdots, (f \circ p_{n}) \times (f^{'} \circ \overline{p}_{n})) \\
		&=& D((f \times f^{'}) \circ (p_{1} \times \overline{p}_{1}), \cdots, (f \times f^{'}) \circ (p_{n} \times \overline{p}_{n})) \\
		&=& D((f \times f^{'}) \circ p_{1}^{'}, \cdots, (f \times f^{'}) \circ p_{n}^{'}),
	\end{eqnarray*}
    where $p_{i}^{'} : X^{n} \times \overline{X}^{n} \rightarrow X \times \overline{X}$ is $ith-$projection map. This shows that \linebreak $k+l \geq$ $TC_{n}(f \times f^{'})$.
\end{proof}

\begin{proposition}
	Let $X \times X$ be normal. For a fibration $f : X \rightarrow X^{'}$, we have TC$(f) \leq$ TC$(X) + 1$.
\end{proposition}

\begin{proof}
	By Proposition \ref{p8} (v), we get
	\begin{eqnarray*}
		TC(f) = D(f \circ p_{1}, f \circ p_{2}) \leq D(p_{1},p_{2}) + D(f,f) = TC(X) + 1.
	\end{eqnarray*} 
\end{proof}

\begin{proposition}
	Given two fibrations $f, f^{'} : X \rightarrow Y^{n}$ ($f \simeq f^{'}$) homotopic to each other, we have TC$_{n}(f) =$ TC$_{n}(f^{'})$.
\end{proposition}

\begin{proof}
	$f \simeq f^{'}$ implies that $f \circ p_{i} \simeq f^{'} \circ p_{i}$ for $i = 1, \cdots, n$. Therefore, we find
	\begin{eqnarray*}
		D(f \circ p_{1}, \cdots, f \circ p_{n}) = D(f^{'} \circ p_{1}, \cdots, f^{'} \circ p_{n}) .
	\end{eqnarray*}
	Thus, we conclude that $TC_{n}(f) = TC_{n}(f^{'})$.
\end{proof}

\quad Note that Proposition 3.14 of \cite{VirgosLois:2020} can be easily generalized as follows:

\begin{proposition}\label{p2}
	Given any homotopy equivalences $\alpha : Y \simeq \overline{Y}$ and $\beta : \overline{X} \simeq X$, assume that $\alpha \circ f_{i} \circ \beta \simeq f^{'}_{i}$ for every $i = 1, \cdots, n$, where $f_{i} : X \rightarrow Y$ and $f^{'}_{i} : \overline{X} \rightarrow \overline{Y}$, i.e., the following diagram commutes for each $i$:
	\begin{displaymath}
	\xymatrix{
		X \ar[r]^{f_{i}} &
		Y \ar[d]^{\alpha} \\
		\overline{X} \ar[u]^{\beta} \ar[r]_{f^{'}_{i}} & \overline{Y}.}
	\end{displaymath}
	Then $D(f_{1}, \cdots, f_{n}) = D(f^{'}_{1}, \cdots, f^{'}_{n})$.
\end{proposition}

\begin{corollary}
	TC$_{n}(f)$ is a fiber homotopy equivalent invariant.
\end{corollary}

\begin{proof}
	Let $f = (f_{1}, \cdots, f_{n}) : X \rightarrow Y^{n}$ and $f^{'} = (f^{'}_{1}, \cdots, f^{'}_{n}) : \overline{X} \rightarrow Y^{n}$ be any fibrations for which there are two functions $u : X \rightarrow \overline{X}$ and $v : \overline{X} \rightarrow X$. These satisfy $u \circ v \simeq 1_{\overline{X}}$ and $v \circ u \simeq 1_{X}$. For each $i \in \{1, \cdots, n\}$, define $\alpha = 1_{Y^{n}}$ and $\beta = d_{n} \circ v \circ \overline{p}_{i}$, where $d_{n} : X \rightarrow X^{n}$ denotes the diagonal function and $\overline{p}_{i} : \overline{X}^{n} \rightarrow \overline{X}$ is the projection. It follows that $\alpha$ and $\beta$ are two homotopy equivalences. Indeed, $\overline{d_{n}} \circ u \circ p_{i}$ is the homotopy inverse of $\beta$, where $\overline{d_{n}} : \overline{X} \rightarrow \overline{X}^{n}$ denotes the diagonal function and $p_{i} : X^{n} \rightarrow X$ is the projection. Then the diagram
	\begin{displaymath}
	\xymatrix{
		X^{n} \ar[r]^{f \circ p_{i}} &
		Y^{n} \ar[d]^{\alpha} \\
		\overline{X}^{n} \ar[u]^{\beta} \ar[r]_{f^{'} \circ \overline{p}_{i}} & Y^{n}}
	\end{displaymath}
	commutes for each $i$. By Proposition \ref{p2}, we get
	\begin{eqnarray*}
		D(f \circ p_{1}, \cdots, f \circ p_{n}) = D(f^{'} \circ \overline{p}_{1}, \cdots, f^{'} \circ \overline{p}_{n}).
	\end{eqnarray*}
    This proves that $TC_{n}(f) = TC_{n}(f^{'})$ for two fiber homotopy equivalent fibrations $f$ and $f^{'}$.
\end{proof}

\begin{proposition}\label{p3}
	TC$_{n}(f) \leq$ TC$_{n}(X)$ for a fibration $f = (f_{1}, \cdots, f_{n}) : X \rightarrow Y^{n}$.
\end{proposition}

\begin{proof}
	The proof follows directly from Proposition \ref{p9} (ii).
\end{proof}

\begin{proposition}\label{p4}
	Let $n>1$. Then TC$(f) \leq$ TC$_{n}(X)$ for a fibration $f : X \rightarrow Y$.
\end{proposition}

\begin{proof}
	By Proposition \ref{p9} (i) and (ii), we obtain
	\begin{eqnarray*}
		TC(f) = D(f \circ p_{1},f \circ p_{2}) \leq D(p_{1},p_{2}) \leq D(p_{1}, \cdots, p_{n}) = TC_{n}(X).
	\end{eqnarray*}
\end{proof}

\quad We state a confirmation for the following well-known result \cite{BasGonRudTam:2014}:

\begin{theorem}\label{t3} 
	Let $X$ be a path-connected space. Then we have \begin{eqnarray*}
		\text{cat}(X^{n-1}) \leq \text{TC}_{n}(X) \leq \text{cat}(X^{n}).
	\end{eqnarray*}
\end{theorem}

\begin{proof}
	Let $i_{k} : X \rightarrow X^{n-1}$ be the inclusion map such that $kth-$component of $i_{k}(x)$ is $x^{'} \in X^{n-1}$, and the other components are always $x$ for $1 \leq k \leq n-1$. Then we find
	\begin{eqnarray*}
		cat(X^{n-1}) = D(1_{X^{n-1}},x^{'}) = D(p_{2} \circ i_{1},p_{1} \circ i_{1}) \leq D(p_{1},p_{2}). 
	\end{eqnarray*}
    Proposition \ref{p9} (i) says that
    \begin{eqnarray*}
    	D(p_{1},p_{2}) \leq D(p_{1}, \cdots, p_{n}) = TC_{n}(X),
    \end{eqnarray*}
    which proves the first inequality. On the other hand, consider the projection maps $p_{1}, \cdots, p_{n} : X^{n} \rightarrow X$ on the path-connected space $X$. By Proposition \ref{p9} (iii), we have
    \begin{eqnarray*}
    	TC_{n}(X) = D(p_{1}, \cdots, p_{n}) \leq cat(X^{n}),
    \end{eqnarray*}
    which completes the proof.
\end{proof}

\begin{corollary} \label{c1}
	If $X$ or $Y$ is contractible, then TC$_{n}(f) = 1$ for a fibration \linebreak$f = (f_{1}, \cdots, f_{n}) : X \rightarrow Y^{n}$.
\end{corollary}

\begin{proof}
	Let $X$ be contractible. Then $cat(X^{n}) = 1$. From Theorem \ref{t3}, we obtain $TC_{n}(X) = 1$. By Proposition \ref{p3}, we conclude that $TC_{n}(f) = 1$. Let $Y$ be contractible. Hence, $Y^{n}$ is contractible. In a similar manner, we get $TC_{n}(Y^{n}) = 1$. By Proposition \ref{p9} (iv), we find
	\begin{eqnarray*}
		TC_{n}(f) = D(f \circ p_{1}, \cdots, f \circ p_{n}) \leq TC_{n}(Y^{n}) = 1.
	\end{eqnarray*} 
    As a result, $TC_{n}(f) = 1$ when $Y$ is contractible.
\end{proof}

\begin{example}
	Let $f_{i} : X \rightarrow \{y_{0}\}$ be a fibration for each $i \in \{1,\cdots,n\}$, where $y_{0}$ is any point of $Y$. Then, by Corollary \ref{c1}, we get $TC_{n}(f) = 1$ for a fibration \linebreak $f = (f_{1}, \cdots, f_{n}) : X \rightarrow \{y_{0}\}$.
\end{example} 
 
\begin{corollary}
	Let $n>2$. For a fibration $f : X \rightarrow Y$ with path-connected $X$ and $Y$, we have 
	\begin{eqnarray*}
		TC(f) \leq \text{cat}(X^{n-1}) \leq TC_{n}(X) \leq \text{cat}(X^{n}).
	\end{eqnarray*}
\end{corollary}

\begin{proof}
	By Proposition \ref{p4} and Theorem $3.16$, it is enough to prove that \linebreak $TC(f) \leq$ $cat(X^{n-1})$ for $n>2$. Proposition \ref{p9} (iii) gives us that
	\begin{eqnarray*}
		TC(f) = D(f^{'} \circ p_{1}, f^{'} \circ p_{2}) \leq cat(X^{2}) \leq cat(X^{n-1})
	\end{eqnarray*}
    for a fibration $f^{'} = (f_{1},f_{2}) : X \rightarrow Y^{2}$.
\end{proof}

\begin{example}
	Let $X$ be any space such that $f = (p_{1}, \cdots, p_{n}) : X^{n} \rightarrow X^{n}$ is a fibration whose components are projections $p_{i} : X^{n} \rightarrow X$ for $i \in \{1,\cdots,n\}$. Then, by using Proposition \ref{p3}, and Corollary 3.14 of \cite{BorVer:2021}, we get
	\begin{eqnarray*}
		TC_{n}(f) \leq TC_{n}(X^{n})
		\leq TC_{n}(X) + TC_{n}(X) + \cdots + TC_{n}(X) = n\cdot TC_{n}(X). 
	\end{eqnarray*}
\end{example}

\section{The Higher Topological Complexity Using Schwarz Genus}
\label{sec:3}

\quad Besides the higher homotopic distance, we also interpret $TC_{n}(f)$ by using the well-known concept Schwarz genus of a fibration.

\begin{lemma}
    For a surjective fibration $f = (f_{1}, \cdots, f_{n}) : X \rightarrow Y^{n}$, the induced map
	\begin{eqnarray*}
		&&e_{n}^{f} : X^{J_{n}} \longrightarrow Y^{n}\\
		&&\hspace*{0.8cm} \alpha \longmapsto e_{n}^{f}(\alpha) = (f_{1}(\alpha_{1}(1)), \cdots, f_{n}(\alpha_{n}(1)))
	\end{eqnarray*}
    is a fibration. 
\end{lemma}

\begin{proof}
	Let $e_{n}^{f} = (f_{1} \times f_{2} \times \cdots \times f_{n}) \circ e_{n}$. Since $f$ is a fibration, for each $i$, $f_{i}$ is a fibration. Then $f_{1} \times f_{2} \times \cdots \times f_{n}$ is a fibration. The composition of two fibrations is again a fibration. This completes the proof.
\end{proof}

\begin{definition}\label{d2}
	For a fibration $f = (f_{1}, \cdots, f_{n}) : X \rightarrow Y^{n}$, $TC_{n}(f)$ is defined as $secat(e_{n}^{f} : X^{J_{n}} \rightarrow Y^{n})$.
\end{definition}

\begin{remark}
	 The first observation is TC$_{1}(f) = 1$. Also, we obtain the equality TC$_{n}(f) =$ TC$_{n}(Y)$ when $f$ is the identity $Y \rightarrow Y$. The next proof \cite{Rudyak:2010} confirms Proposition \ref{p6}:
\end{remark}

\textbf{Proof of Proposition 3.9:}
	Let $TC_{n+1}(f) = k$. Then $U_{1}, \cdots, U_{k}$ is an open cover of $Y^{n+1}$ and $s_{i} : U_{i} \rightarrow X^{J_{n+1}}$ is a section of $e_{n+1}^{f}$ for all $i = 1, \cdots, k$. Let \[V_{i} = \{(y_{1}, \cdots, y_{n}) \ : \ (y_{1}, \cdots, y_{n}, a) \in Y^{n+1}\} \subset U_{i}\] for $a \in Y$. Define two maps $h : T_{n+1}(Y) \rightarrow T_{n}(Y)$, $h(\alpha_{1}, \cdots, \alpha_{n+1}) = (\alpha_{1}, \cdots, \alpha_{n})$ and $k : Y^{n} \rightarrow Y^{n+1}$, $k(y_{1}, \cdots y_{n}) = (y_{1}, \cdots, y_{n},a)$, where $T_{n}(Y)$ is a set consists of ordered set of $n$ paths in $Y$. It follows that $V_{1}, \cdots, V_{k}$ is an open cover of $Y^{n}$ and $t_{i} = h \circ s_{i} \circ k$ is a section of $e_{n}^{f}$ on $V_{i}$ for each $i$. Thus, $TC_{n}(f) \leq k$.\\
	\hspace*{12.4cm}$\square$
	
\begin{lemma}\label{l1}
	Given any fibrations $f^{'} : X \rightarrow Y$ and $f = (f_{1},f_{2}) : X \rightarrow Y \times Y$, we have $secat(\pi_{f^{'}} : X^{I} \rightarrow X \times Y) = secat(e_{2}^{f} : X^{J_{2}} \rightarrow Y \times Y)$.
\end{lemma}

\begin{proof}
	Let $\alpha$ be a path in $X^{J_{2}}$ such that $\alpha(0) = x_{0}$, $\alpha_{1}(1) = x_{1}$ and $\alpha_{1}(1) = x_{2}$. Define $h : X^{J_{2}} \rightarrow X^{I}$, $h(\alpha) = \beta$, where $\beta$ is a path from $x_{1}$ to $x_{2}$. $h^{'} : X \times X \rightarrow X^{I}$, $h(x,x^{'})$ denotes a path, its starting point is $x$ and its final point is $x^{'}$. Therefore, $h$ is a fiber homotopy equivalence. Indeed, consider two maps $u : X^{J_{2}} \rightarrow X \times X$, $u(\alpha) = (\alpha_{1}(1),\alpha_{2}(1))$ and $v : X \times X \rightarrow X^{J_{2}}$, $v(x,x^{'}) = \beta$, where $\beta$ is a path from any point of $X$ to two points $\beta_{1}(1) = x$, $\beta_{2}(1) = x^{'}$. Then
	\begin{displaymath}
	\xymatrix{
		X^{J_{2}} \ar[dr]_{h} \ar@<1ex>[rr]^u
 & & X \times X  \ar@<1ex>[ll]^v \ar[dl]^{h^{'}} \\
		& X^{I} & }
	\end{displaymath}
	commutes such that $u \circ v$ and $v \circ u$ are homotopic to respective identity maps $1_{X \times X}$ and $1_{X^{J_{2}}}$. Recall that $f_{1}$, $f_{2}$ and $f^{'}$ are surjective maps. Using this fact, we define a map $k : Y \times Y \rightarrow X \times Y$, $k(y_{1},y_{2}) = (x_{1},f^{'}(x_{2}))$, where $y_{1} = f_{1}(x_{1})$ and $y_{2} = f_{2}(x_{2})$. We shall show that $k$ is one of equivalences of homotopic functions. Consider the map $k^{'} : X \times Y \rightarrow Y \times Y$, $k^{'}(x,y) = (f_{1}(x),f_{2}(x^{'}))$, where $y = f^{'}(x^{'})$. Then $k \circ k^{'}$ and $k^{'} \circ k$ are respective identity maps $1_{X \times Y}$ and $1_{Y \times Y}$. Finally, by Theorem 6.4 of \cite{Rudyak:2016}, the following diagram gives the desired result:
	\begin{displaymath}
	\xymatrix{
		X^{J_{2}} \ar[r]^{h} \ar[d]_{e_{2}^{f}} &
		X^{I} \ar[d]^{\pi_{f^{'}}} \\
		Y \times Y \ar[r]_{k} & X \times Y.}
	\end{displaymath} 	
\end{proof}

\quad By Lemma \ref{l1}, we have the quick result:

\begin{corollary}
	TC$_{2}(f)$ in Definition \ref{d2} coincides with TC$(f)$ in Definition \ref{d6}.
\end{corollary}

\begin{theorem}
	Given a fibration $f : X \rightarrow Y^{n}$, $e_{n}^{f} : X^{J_{n}} \rightarrow Y^{n}$ is a fiber homotopy equivalent to the function $q : X \sqcap Y^{J_{n}} \rightarrow Y^{n}$ defined by $q(x,\alpha) = (\alpha_{1}(1), \cdots, \alpha_{n}(1))$.
\end{theorem}

\begin{proof}
	Let $\pi_{2} : X \sqcap Y^{J_{n}} \rightarrow Y^{J_{n}}$. Since $q = e_{n} \circ \pi_{2}$, $q$ is a fibration. Define a function $u : X^{J_{n}} \rightarrow X \sqcap Y^{J_{n}}$ by $u(\alpha) = (\alpha(0),(f_{1} \times \cdots \times f_{n}) \circ \alpha)$, where $\alpha$ is a path with endpoints $x_{1}, \cdots, x_{n}$, i.e., $\alpha_{1}(1) = x_{1}$, $\cdots$, $\alpha_{n}(1) = x_{n}$, and $v : X \sqcap Y^{J_{n}} \rightarrow X^{J_{n}}$ by $v(x,\beta) = \alpha$, where $\beta$ is a path from $x$ to $(f_{1}(x_{1}), \cdots, f_{n}(x_{n}))$. Then $u \circ v$ and $v \circ u$ are homotopic to identity maps $1_{X \sqcap Y^{J_{n}}}$ and $1_{X^{J_{n}}}$, respectively. Moreover, the following diagram commutes:
	\begin{displaymath}
	\xymatrix{
		X^{J_{n}} \ar[dr]_{e_{n}^{f}} \ar@<1ex>[rr]^u & & X \sqcap Y^{J_{n}} \ar@<1ex>[ll]^v \ar[dl]^{q} \\
		& Y^{n}. & }
	\end{displaymath}
	This completes the proof.
\end{proof}

\begin{corollary}
	The function $q : X \sqcap Y^{J_{n}} \rightarrow Y^{n}$ is a pullback fibration of \linebreak $e_{n} : X^{J_{n}} \rightarrow X^{n}$. Moreover, TC$_{n}(f) = secat(q)$.
\end{corollary}

\begin{proof}
	Let $h : X \sqcap Y^{J_{n}} \rightarrow Y^{J_{n}}$ be the projection. Then 
	\begin{displaymath}
	\xymatrix{
		X \sqcap Y^{J_{n}} \ar[r]^{h} \ar[d]_{q} &
		Y^{J_{n}} \ar[d]^{e_{n}} \\
		Y^{n} \ar[r]_{1_{Y}} & Y^{n}}
	\end{displaymath}
	commutes. Furthermore, by Theorem 3.24, we get
	\begin{eqnarray*}
		TC_{n}(f) = secat(e_{n}^{f}) = secat(q).
	\end{eqnarray*} 
\end{proof}
\section{Conclusion}
\label{sec:4}
\quad The topological complexity is an essential homotopy invariant for the work of topological robotics. TC$_{n}(X)$ and TC$(f)$ improve this investigation, where $f$ is a continuous and surjective map. The next step is to reveal the expression of TC$_{n}(f)$. We give an answer to this problem when $f$ is a surjective fibration. If we consider the case that $f$ does not have to be a fibration, in other saying, $f$ is just a continuous and surjective map, then this is a still open problem. In this study, with the generalization of TC$(f)$, we have one extra method to examine the problem of motion planning in topological robotics. We also contribute to the investigation of relationship between D$(f,g)$ and secat because the another common point of these two notions is given by TC$_{n}(f)$. One can have a certain fibration $f$ to determine TC$_{n}$ or cat of a space of a fibration by using the properties in Section \ref{sec:2} or Section \ref{sec:3}. The reader is free to choose the way following on D$(f_{1}, \cdots, f_{n})$ for a multipath $f = (f_{1}, \cdots, f_{n}) : X \rightarrow Y^{n}$ or the way following on secat $e_{n}^{f} : X^{J_{n}} \rightarrow Y^{n}$ with $e_{n}^{f} = (f_{1} \times f_{2} \times \cdots \times f_{n}) \circ e_{n}$, where $e_{n}$ is the path fibration in the definition of TC$_{n}$ of a path-connected topological space $X$. 

\quad In digital images, there are many different computations on the notions TC, TC$_{n}$ and cat rather than topological spaces. Analogously, the definition of TC$_{n}(f)$ and the related results can be adapted to the digital images using the digital meaning of the (higher) homotopic distance. The task is to determine the similarities and the differences on TC$_{n}$ of a map between ordinary spaces and digital images.

\quad One of the future researches on TC$_{n}(f)$ is to state its symmetric version, in other stating, the higher symmetric topological complexity for a map (or a fibration). Each of directed and monodial versions of TC$_{n}(f)$ is another topic for the reader. 

\acknowledgment{Research Fund of the Ege University partially supported this work (The Project Number is FDK-2020-21123). Also, the Scientific and Technological Research Council of Turkey TUBITAK-2211-A grants the first author as fellowship.}

\end{document}